\documentclass[12pt,article]{article}
\usepackage[latin1]{inputenc}
\usepackage{amsthm}
\usepackage{amsfonts}
\usepackage{amssymb}
\usepackage{enumerate}

\usepackage{amsmath}

\newtheorem{lemme}{Lemma}
\newtheorem{proposition}{Proposition}
\newtheorem{theorem}{Theorem}
\newtheorem*{theorem*}{Theorem}
\newtheorem{remark}{Remarks}

\DeclareMathOperator{\dive}{div}
\DeclareMathOperator{\sgn}{sgn}

\begin{document}
	\title{$L^1$-stability of periodic stationary solutions of scalar convection-diffusion equations}
	\author{Val\'erie Le Blanc\footnote{Universit\'e de Lyon; Universit\'e Lyon 1; INSA de Lyon, F-69621; Ecole Centrale de Lyon; CNRS, UMR5208, Institut Camille Jordan; 43, boulevard du 11 novembre 1918; F-69622 Villeurbanne-Cedex, France. E-mail address: leblanc@math.univ-lyon1.fr.}}
	\maketitle
	
	\begin{abstract}
		The aim of this paper is to study the $L^1$-stability of periodic stationary solutions of scalar convection-diffusion equations. We obtain dispersion in $L^2$ for all space dimensions using Kru\v{z}kov type entropy. And when the space dimension is one, we estimate the number of sign changes of a solution to obtain $L^1$-stability.
	\end{abstract}


	\section{Introduction}
	
	We study the solutions of a scalar convection-diffusion equation of the form:
	\begin{equation}\label{equation}\partial_{t}u+\dive (f(u,x))=\Delta u,\quad t>0,x\in\mathbb R^d,\end{equation}
	where $x\mapsto f(\cdot,x)$ is an $Y$-periodic function with $Y=\prod_{i=1}^d(0,T_{i})$ the basis of a lattice. We assume that $f$ belongs to $\mathcal C^2(\mathbb R,\mathcal C^1(\mathbb R^d))$. For this equation, periodic stationary solutions $w_p$ exist and are parameterized by their space average $p$: this is a result of Dalibard in \cite{Dalibard}. In this paper, we focus on the $L^1$-stability of these periodic stationary solutions.

	When $f$ only depends on $u$, the periodic stationary solutions are the constants and the $L^1$-stability of the constants is already proved by Freist\"uhler and Serre in the one-dimensional space case in \cite{Freistuhler} and by Serre in all space dimension in \cite{Serre}. We define the space $$L_0^1(\mathbb R^d)=\{u\in L^1(\mathbb R^d) : \int_{\mathbb R^d}u(x)dx=0\}.$$ With this notation, the result can be written as follows:
	\begin{theorem}
		\emph{\cite{Serre}} For all $k\in\mathbb R, b\in L_0^1(\mathbb R^d)\cap L^\infty(\mathbb R^d)$, the unique solution $u\in L^\infty_{loc}(\mathbb R,L^\infty(\mathbb R^d))$ of
		\begin{equation}\label{eqSerre}\left\{\begin{array}{lr}\partial_{t}u+\dive (f(u))=\Delta u,&t>0,x\in\mathbb R^d,\\u(0,x)=k+b(x),&x\in\mathbb R^d,\end{array}\right.\end{equation}
		satisfies:$$\lim_{t\to\infty}\|u(t,\cdot)-k\|_{1}=0.$$
	\end{theorem}
	
	The proof of this result can be made in 3 steps. First, the global existence of solution of (\ref{eqSerre}) is proved using the Duhamel's formula with $\dive(f(u))$ as a perturbation of the Heat equation, one obtains : 
	$$u(t)=K^t*u_0+\int_0^t\dive K^{t-s}*f(u(s))ds.$$
	The maximum principle allows to conclude about global existence by induction. This defines the nonlinear semigroup $\tilde S^t$ so that $u(t)=\tilde S^tu_0$ is the solution of (\ref{eqSerre}).
	
	Secondly, one establishes the so-called four ``Co-properties'' for $u_0,v_0$ in $L^\infty(\mathbb R^d)$:
	\begin{enumerate}
		\item Comparison: $u_0\leq v_0$ a.e.$\Rightarrow \tilde S^tu_0\leq\tilde S^tv_0$ a.e.,
		\item Contraction: $v_0-u_0\in L^1(\mathbb R^d)\Rightarrow \tilde S^tv_0-\tilde S^tu_0\in L^1(\mathbb R^d)$ and $$\|\tilde S^tv_0-\tilde S^tu_0\|\leq\|v_0-u_0\|,$$
		\item Conservation (of mass): $v_0-u_0\in L^1(\mathbb R^d)\Rightarrow \tilde S^tv_0-\tilde S^tu_0\in L^1(\mathbb R^d)$ and $$\int_{\mathbb R^d}(\tilde S^tv_0-\tilde S^tu_0)=\int_{\mathbb R^d}(v_0-u_0),$$
		\item Constants: if $u_0$ is a constant, then $\tilde S^tu_0\equiv u_0$.
	\end{enumerate}
	 
	Two methods allow to conclude: one in one space dimension and another one in all space dimension. The first one is due to Freist\"uhler and Serre \cite{Freistuhler}: they study the number of sign changes of the solution. Having assumed that $k=0, f(0)=0$, they study the primitive $V$ of the solution $u$ which vanishes at $-\infty$: $V(x,t)=\int_{-\infty}^xu(y,t)dy$. Since $b\in L^1_0(\mathbb R),$ this primitive also vanishes at $+\infty$ and belongs to $L^\infty(\mathbb R)$. Moreover, $V$ satisfies a parabolic equation $$\partial_tV+f(\partial_xV)=\partial_x^2V.$$ They also apply the lemma of Matano \cite{Matano} on $V$ to estimate the number of sign changes of the derivative of $V$: $u$. Estimates on both $\|u(t)\|_{L^1}$ by $\|V(t)\|_{L^\infty}$ follow. Using $L^2$-estimates on the equations on both $u$ and $V$, one shows that $\displaystyle\lim_{t\to\infty}\|V(t)\|_{L^\infty}=0$, which permits to obtain the theorem.
	
	The second method, due to Serre \cite{Serre}, is based on the Duhamel's formula. A dispersion inequality is obtained using the entropy $u\mapsto u^2$ for equation (\ref{eqSerre}) and $L^1$-contraction, one obtains :
	$$\|\tilde S^tu_0\|_{2}\leq c_{d}\frac{\|u_0\|_{1}}{t^{d/4}}.$$
	Under the rather general assumption that $f(u)$ is bounded by $|u|^2$,we prove $\displaystyle\lim_{t\to\infty}\|\tilde S^tb\|_1=0$ combining dispersion estimate and estimates on the heat kernel.
	
	In this article, we will see how we can adapt some of these arguments to the case where $f$ depends both on $u$ and $x$. We recall that in this case the stationary solutions $w_p$ considered are periodic, parameterized by their space average $p$. 
	
	We obtain one theorem in the one-dimensional space case :
	\begin{theorem}
		For all $p\in\mathbb R, b\in L_0^1(\mathbb R)\cap L^\infty(\mathbb R)$, the unique solution $u$ in $L^\infty_{loc}(\mathbb R,L^\infty(\mathbb R))$ of
		$$\left\{\begin{array}{lr}\partial_{t}u+\dive (f(u,x))=\Delta u,&t>0,x\in\mathbb R,\\u(0,x)=w_{p}+b(x),&x\in\mathbb R,\end{array}\right.$$
		satisfies:
		$$\lim_{t\to\infty}\|u(t,\cdot)-w_{p}\|_{1}=0.$$\label{thm}
	\end{theorem}
	
	First, we observe that in this theorem we assume $\int_{-\infty}^\infty b(x)dx=0.$ This assumption is necessary because of the conservation of mass: $$\int_{\mathbb R^d}(v-w_p)=\int_{\mathbb R^d}(v_0-w_p)=\int_{\mathbb R^d}b.$$ Actually, we can not have $L^1$-convergence when $\int_{\mathbb R^d}b\neq0.$ But this assumption is not necessary to prove $L^p$-convergence for $1<p\leq2$ and in this case we obtain a rate of convergence $d/2(1-1/p)$.
	
	To prove the theorem, we use results on the nonlinear semigroup and the lemma of Matano, as in \cite{Freistuhler}. The main difference with the proof of Serre and Freist\"uhler (\cite{Serre} \& \cite{Freistuhler}) appears in the proof of $L^2$-estimates for $u$ and its primitive $V$. Since the problem is inhomogeneous, $u\mapsto u^2$ is not an entropy and we have to find a new entropy to prove dispersion inequality. For $V$ the results on periodic stationary solutions of Dalibard permit to prove that $\|V\|_2$ is bounded.
		
	The paper is organized as follows. In section 2, we recall the result obtained by Dalibard in \cite{Dalibard} about the existence of periodic stationary solutions. In section 3, we focus on the existence and the properties of our nonlinear semigroup in all space dimension: comparison principle, contraction in $L^1$, conservation of mass, dispersion inequality. For its existence and its three first properties the proofs are similar to the homogeneous case $f(u,x)=f(u)$, except that the maximum principle does not hold anymore and is replaced by a comparison principle. For the dispersion inequality, we build a new type of Kru\v{z}kov entropy, based on periodic stationary solutions instead of constants. In section 4, we focus on the one-dimensional space case, and prove theorem \ref{thm} using the lemma of Matano about the number of sign changes.
	
	\section{Existence of stationary solutions}
	In this section, we recall the existence result of Dalibard \cite{Dalibard}. When $f$ depends only on $u$, but not on $x$, i.e. when we are in the case studied by Serre in \cite{Serre}, the stationary solutions considered are all the constants. But in our case the constants are not solutions except if $\dive (f(k,x))=0$ for all $x \in\mathbb R^d$. The existence of another class of stationary solutions is proved by Dalibard (see theorem 2 and lemma 6 in \cite{Dalibard}): there exist periodic stationary solutions, indexed by their space average.
	
	In this section, we recall a part of her results for the following equation:
	$$\dive (f(u,x))=\Delta u, x\in\mathbb R^d$$ where $x\mapsto f(\cdot,x)$ is an $Y$-periodic function with $Y=\prod_{i=1}^d(0,T_{i})$ the basis of a lattice.
	We note the space average of a function $u$: $\langle u\rangle_Y=\frac1{|Y|}\int_Yu(x)dx.$
	\begin{theorem}\label{thmdalibard}
		Let $f=f(u,x)\in \mathcal C^2(\mathbb R,\mathcal C^1(\mathbb R^d))$ such that $\partial_uf\in L^\infty(\mathbb R\times Y)$. Assume that there exist $C_{0}>0,$ and $n\in[0,\frac{d+2}{d-2})$ when $d\geq 3$, such that for all $(p,x)\in \mathbb R\times Y$
		$$|\dive f(p,x)|\leq C_{0}(1+|p|^n).$$
		Then for all $p\in\mathbb R$, there exists a unique solution $w(\cdot,p)\in H^1_{per}(Y)$ of
		$$-\Delta w(x,p)+\dive f(w(x,p),x)=0,\text{ such that }\langle w(\cdot,p)\rangle_{Y}=p.$$
		For all $p\in\mathbb R,w(\cdot,p)$ belongs to $W_{per}^{2,q}(Y)$ for all $1<q<\infty$ and for all $R>0$, there exists $C_R>0$ such that
		$$\|w(\cdot,p)\|_{W^{2,q}(Y)}\leq C_{R}\quad\forall p\in\mathbb R,|p|\leq R,$$
		$C_{R}>0$ depending only on $d,Y,C_{0},n,q,p_{0}$ and $R$.
		
		Furthermore, for all $p\in\mathbb R,\partial_{p}w(\cdot,p)\in H_{per}^1(Y)$ is in the kernel of the linear operator
		$$-\Delta+\dive (\partial_uf(w(x,p),x)\cdot)=0,\text{ and }\langle\partial_{p}w\rangle_{Y}=1.$$
		And there exists $\alpha>0$ depending only on $d,Y$ and $\|\partial_uf\|_\infty$ such that $$\partial_{p}w(x,p)>\alpha\text{ for a.e. }(x,p)\in Y\times\mathbb R.$$
		Hence, $$\lim_{p\to+\infty}\inf_Y w(x,p)=+\infty,$$
		$$\lim_{p\to-\infty}\sup_Y w(x,p)=-\infty.$$
	\end{theorem}

	\begin{remark}\label{remark}~
{\em		\begin{itemize}
			\item A consequence of this theorem is that for all $x\in\mathbb R^d$, the application $p\mapsto w(p,x)$ is increasing and bijective from $\mathbb R$ to $\mathbb R$.
			\item In this theorem, we impose the restrictive assumption that $\partial_uf\in L^\infty$ on the whole domain $\mathbb R\times Y$. When $\partial_uf$ belongs only to $L_{loc}^\infty (L^\infty(Y))$, we obtain that $\partial_pw>0$ but we have not the existence of the constant $\alpha$. Hence, we have no result on the limit when $p\to\pm\infty$ of $\inf_Y w(x,p)$ and $\sup_Y w(x,p)$, but we have that the application 
			$$\begin{array}{rcl}\mathbb R&\to&\displaystyle\big]\lim_{p\to+\infty}\inf_Y w(x,p),\lim_{p\to-\infty}\sup_Y w(x,p)\big[\\p&\mapsto&w(p,x)\end{array}$$ is bijective. And we can adapt the result of theorem \ref{thm} in this case : we just have to make the assumption that there exists $p$ such that for all $x\in\mathbb R^d, u_0(x) \in[w(-p,x),w(p,x)]$.
		\end{itemize}}
	\end{remark}

	In the sequel, we use the notation: $w_{p}=w(\cdot,p)$.
	
	\section{The nonlinear semigroup}
	
	In what follows, we focus on the Cauchy problem for equation (\ref{equation}):
	\begin{equation}\label{existence}\left\{\begin{array}{lr}\partial_{t}u+\dive (f(u,x))=\Delta u,& \forall t>0,\forall x\in\mathbb R^d,\\u(0,x)=u_0(x),&x\in\mathbb R^d,\end{array}\right.\end{equation}
	where the initial datum $u_0$ belongs to $L^\infty(\mathbb R^d).$ First, we adapt the approach of Serre \cite{Serre} to prove the existence of solutions and their properties: comparison principle, $L^1$-contraction, conservation of mass. Then, we prove a dispersion inequality, using a new type of entropy based on periodic solutions.
	
	\subsection{Existence of the nonlinear semigroup}
	As in \cite{Serre}, the proof of the existence of solutions is based on Duhamel's formula for heat equation. We also need a comparison principle to replace the maximum principle which is not true here.
	
	Let us write problem (\ref{existence}) in the form:
	\begin{equation}\label{existence2}\left\{\begin{array}{lr}\partial_{t}u-\Delta u=-\dive (f(u,x)),& t>0,x\in\mathbb R^d,\\u(0,x)=u_0(x),&x\in\mathbb R^d.\end{array}\right.\end{equation}
	Here, the heat operator appears in the left handside of (\ref{existence2}), and the right handside is a lower order perturbation. Denote $H^t$ the heat semigroup and $K^t$ its kernel. They are given by:
	$$H^tu_0=K^t*u_0,\quad K^t(x)=\frac1{(2\pi t)^{d/2}}\exp\left(-\frac{\|x\|^2}{4t}\right)$$
	and satisfy the following properties:
	\begin{eqnarray}\label{heatinequality1}\|H^tu_0\|_{p}\leq\|u_0\|_{p},&1\leq p\leq\infty,\\	\label{heatinequality2}\|\nabla_{x}H^tu_0\|_{p}\leq c'_{p}t^{-\frac12}\|u_0\|_{p},&1\leq p\leq\infty, \end{eqnarray}
	\begin{equation}\label{integral}\int_{\mathbb R^d}K^t(x)dx=1,\quad\int_{\mathbb R^d}\nabla_{x}K^t(x)dx=0.\end{equation}
	
	We prove the following result:
	\begin{proposition}\label{local}
		Assume that $f\in\mathcal C^k(\mathbb R,\mathcal C^1(\mathbb R^d)).$ Then for all $a\in L^\infty(\mathbb R^d)$, there exist $T>0$ and a unique solution $u\in L^\infty([0,T]\times\mathbb R^d)$ of (\ref{existence}). Moreover, $u\in \mathcal{C}^k((0,T),\mathcal{C}^\infty(\mathbb{R}^d))$ and $T$ depends only on $\|u_0\|_\infty$.
	\end{proposition}
	\begin{proof}
		We are searching for the mild solution of (\ref{existence}), i.e which verifies the Duhamel's formula:
		$$\begin{array}{rcl}u(t,\cdot)&=&K^t*u_0-\int_{0}^tK^{t-s}*\dive (f(u(s,\cdot),\cdot))ds\\&=&K^t*u_0-\int_{0}^t\nabla_{x}K^{t-s}*f(u(s,\cdot),\cdot)ds.\end{array}$$
		Hence, we search for the solution of (\ref{existence}) as a fixed point of the map
		$$M:u\mapsto\Big(t\mapsto K^t*u_0-\int_{0}^t\nabla_{x}K^{t-s}*f(u(s,\cdot),\cdot)ds\Big).$$
		In order to use Picard's fixed point theorem we need to find a space which is stable by $M$ and where $M$ is contractant. Using (\ref{heatinequality1})-(\ref{heatinequality2}) with $p=\infty$ we have the following estimate for all $u\in L^\infty(\mathbb R^d)$:
		$$\|Mu(t)\|_\infty\leq \|u_0\|_{\infty}+\int_{0}^t\frac{c'_{\infty}}{(t-s)^{\frac12}}\|f(u(s,\cdot),\cdot)\|_{\infty}ds.$$
		We assume that for all $0\leq~s\leq~T, \|u(s)\|_{\infty}\leq2\|u_0\|_{\infty}$. Since $f(\cdot,x)$ is locally in $L^\infty$, uniformly in $x$, there exists $C$ such that for all $0\leq s\leq T,$ $$\|f(u(s,\cdot),\cdot)\|_{\infty}\leq C$$ where $C$ does not depend on $u$, but only on $\|u\|_{L^\infty((0,t)\times\mathbb R^d)}\leq 2\|u_0\|_{\infty}.$ Therefore, we obtain the following estimate
		$$\|Mu(t)\|_\infty\leq \|u_0\|_{\infty}+2c'_{\infty}C\sqrt T, \quad\forall0\leq t\leq T.$$
		For $T$ sufficiently small ($2c'_{\infty}C\sqrt T<\|u_0\|_{\infty}$), the map $M$ preserves the ball of radius $2\|u_0\|_{\infty}$ of $L^\infty((0,T)\times\mathbb R^d)$. This ball is denoted $B(2\|u_0\|_{\infty})$. Next we prove that $M$ is a contraction: let $u,v \in B(2\|u_0\|_{\infty})$, then
		$$Mv(t)-Mu(t)=\int_{0}^t\nabla_{x}K^{t-s}*(f(u(s,\cdot),\cdot)-f(v(s,\cdot),\cdot))ds.$$
		Since $f(\cdot,x)$ is locally Lipschitz, uniformly in $x$, there exists $C'$ (depending on $2\|u_0\|_{\infty}$) such that $\|f(u,\cdot)-f(v,\cdot)\|_{\infty}\leq C'\|u-v\|_{\infty}$. Hence, we obtain
		$$\|Mu-Mv\|_\infty\leq 2c'_{\infty}C'\sqrt T\|u-v\|_{\infty}$$
		and for $T$ small enough, the map $M$ is stable and contractant on $B(2\|u_0\|_{\infty}).$\\
		We can now use Picard's fixed point theorem to obtain a unique local solution in $L^\infty([0,T]\times\mathbb R^d)$. Moreover, using again Duhamel's formula, we prove that this solution is regular in time if $f$ is regular in $u$ and $x$ ; for instance $u$ is in $\mathcal C^k((0,T),\mathcal C^\infty(\mathbb R^d))$ if $f$ is in $C^k(\mathbb R,\mathcal C^1(\mathbb R^d))$.	
	\end{proof}
		
	To prove global existence in homogeneous problem, one uses maximum principle. When the problem is inhomogeneous, this maximum principle is false and one uses a comparison principle:
	\begin{lemme}\emph{Comparison principle:}\label{comparison}
		Let $u,v\in L^\infty([0,T]\times\mathbb R^d)$ two solutions of (\ref{equation}) on $(0,T)$ such that for all $x\in\mathbb R^d,u_0(x)\leq v_0(x)$. Then for all $t\in[0,T]$, and $x\in\mathbb R^d$, we have $u(t,x)\leq v(t,x).$
	\end{lemme}	
	
	Using this lemma, we then prove global existence of solution:
	\begin{proposition}
		Assume that $f\in\mathcal C^k(\mathbb R,\mathcal C^1(\mathbb R^d)).$ Then for all $u_0\in L^\infty(\mathbb R^d)$, there exists a unique solution $u\in \mathcal{C}^k(\mathbb R,\mathcal{C}^\infty(\mathbb{R}^d))$ of (\ref{existence}).
	\end{proposition}
	\begin{proof}
		From theorem \ref{thmdalibard} and the remark \ref{remark} we deduce that for all $x$, the application $p\mapsto w_p(x)$  is invertible from $\mathbb R$ to $\mathbb R$. Since $u_0\in L^\infty(\mathbb R^d)$, there exists $p$ such that $w_{-p}(x)\leq u_0(x)\leq w_p(x)$. Proposition \ref{local} gives us $T$ (we can chose $T=T(\max\{\|w_{-p}\|_\infty,\|w_p\|_\infty\})$) and a unique solution $u$. The lemma implies that for all $t\in(0,T)$, and $x\in\mathbb R$, we have $w_{-p}(x)\leq u(t,x)\leq w_p(x)$. Therefore, we can iterate the local existence to prove that $u$ exists on $(0,T),\dots,(kT,(k+1)T)$ for any $k\in\mathbb N$. Finally, we obtain a unique bounded solution, global and smooth for positive time.
	\end{proof}	

	Next, we define the nonlinear semigroup $S^t$ on $L^\infty(\mathbb R^d)$. From now, we will note $u=S^tu_0, v=S^tv_0$ if $u_0,v_0\in L^\infty(\mathbb R^d).$

	As in \cite{Serre}, we have some properties on this semigroup: we have already mentioned the comparison principle (lemma \ref{comparison}). We also have $L^1$-contraction and conservation of mass. And as said above, the constants are no longer stationary solutions: they are replaced by periodic functions.
	\begin{proposition}For all $u_0,v_0\in L^\infty(\mathbb R^d)$ such that $u_0-v_0\in L^1(\mathbb R^d)$, for all $t>0$ we have 
		\begin{enumerate}[i)]
			\item \emph{$L^1$-contraction :} $S^tu_0-S^tv_0\in L^1(\mathbb R^d)$ and $\|S^tu_0-S^tv_0\|_1\leq\|u_0-v_0\|_1;$
			\item \emph{conservation of mass :} $\displaystyle\int_{\mathbb R^d}(S^tu_0-S^tv_0)=\int_{\mathbb R^d}(u_0-v_0).$
		\end{enumerate}
	\end{proposition}
	\begin{proof}
		Let $u_0,v_0\in L^\infty(\mathbb R^d)$ such that $u_0-v_0\in L^1(\mathbb R^d)$. We first prove that $S^tu_0-S^tv_0\in L^1(\mathbb R^d).$ Using Duhamel's formula, one obtains:
		\begin{equation}\label{relation}v(t)-u(t)=K^t*(v_0-u_0)-\int_{0}^t(\nabla_{x}K^{t-s})*(f(v(s,\cdot),\cdot)-f(u(s,\cdot),\cdot))ds.\end{equation}
		Taking the $L^1$-norm and using estimates (\ref{heatinequality1})-(\ref{heatinequality2}) for $p=1$, we deduce that
		$$\sup_{s\leq t}\|v(s)-u(s)\|_{1}\leq\|v_0-u_0\|_{1}+2c'_{1}C'\sqrt t\sup_{s\leq t}\|v(s)-u(s)\|_{1}.$$
		Hence, for $t$ small enough, $v(s)-u(s)\in L^1(\mathbb R^d),$ for all $0\leq s\leq t$ and by induction it is true for all $t\in \mathbb R^+.$\\
		We now prove the $L^1$-contraction principle. For all $u_0,v_0\in L^\infty(\mathbb R^d)$ one shows that $$\partial_{t}|u-v|+\dive (\sgn (u-v)(f(u,\cdot)-f(v,\cdot)))\leq \Delta |u-v|.$$
		Noting
		\begin{equation}\label{w}w=-K^t*|v_0-u_0|+\int_{0}^t\partial_{x}K^{t-s}*\dive ((f(u,x)-f(v,x))\sgn (u-v))+|u-v|,\end{equation} we easily prove $\partial_{t}w\leq\Delta w$ and $w(0)=0$. Using comparison principle, we have $w\leq0$. We integrate (\ref{w}) according to $x$ to obtain \begin{equation}\label{egint}0\geq\int_{\mathbb R^d}w=-\int_{\mathbb R^d}|v_0-u_0|+\int_{\mathbb R^d}|u-v|.\end{equation} From (\ref{egint}), we deduce the contraction principle.\\
		Let us now prove the conservation of mass. Integrating (\ref{relation}), and using (\ref{integral}) we immediately obtain for all $u_0,v_0\in L^\infty(\mathbb R^d)$: $\partial_{t}\int_{\mathbb R^d}(u-v)=0$ and $$\int_{\mathbb R^d}(u-v)=\int_{\mathbb R^d}(u_0-v_0).$$
	\end{proof}
	
	\subsection{Dispersion inequality}
	
	In this section, we prove the following dispersion inequality for equation (\ref{equation}): 
	\begin{proposition}\label{propdisp}
		Let $R\in\mathbb R.$ There exists $C>0$ so that for all $p\in\mathbb R, b\in L^\infty(\mathbb R^d)\cap L^1(\mathbb R^d)$ such that $w_{-R}\leq w_p+b\leq w_R$, $u(t)=S^t(w_p+b)$ verifies a dispersion inequality:
		\begin{equation}\label{dispersion}\|u(t)-w_p\|_2\leq C_d\frac{\|b\|_1}{t^{d/4}}.\end{equation}
	\end{proposition}
This estimate gives convergence in $L^2$ when $u_0-w_0\in L^1(\mathbb R^d)$ and the speed of this convergence. In section \ref{oneD}, we will see how $L^2$-convergence imply $L^1$-convergence in the one dimensional space case.
	
This property is first proved by B\'enilan and Abourjaily in \cite{AbBe} in the case where $f$ does not depend on $x$. When $\tilde S^t$ denotes the semigroup of (\ref{eqSerre}), their result can be written as follows:
	 $$\|\tilde S^tu_0\|_{2}\leq c_{d}\frac{\|u_0\|_{1}}{t^{d/4}}.$$
	  In this case, the proof of the inequality is based on the fact that for all convex function $\eta$, there exists $g$ such that for all $u$, $\eta'(u)\dive (f(u))=\dive (g(u))$, in particular for $\eta(u)=u^2$. This property is false in our case but we still have a dispersion inequality (\ref{dispersion}).
	  	
To prove proposition \ref{propdisp}, we use a new class of entropies. When $f$ does not depend on $x$, an interesting class of entropies is the Kru\v{z}kov entropies $u\mapsto|u-k|$ with $k\in\mathbb R$. Those are convex functions and for all $u$ solution of (\ref{eqSerre}), we have the inequality
	$$\partial_{t}|u-k|+\dive (\sgn (u-k)(f(u)-f(k)))\leq \Delta |u-k|.$$
	This inequality is still true in our case but we do not want to compare our solutions to constants  anymore, because they are not stationary solutions of (\ref{existence}). Hence, we define a new type of entropy, using the stationary solutions $w_p$.
	
	\begin{proof}
		Without loss of generality we assume that $p=0$. We have just said that we need to base our new entropy on the stationary solutions. Theorem \ref{thmdalibard} gives us that for all $p\in\mathbb R$, there exists a unique stationary solution $w_{p}$ under the constraint  $\langle w_{p}\rangle_{Y}=p.$ Following the construction of Kru\v{z}kov entropies, let us consider , for any $p\in\mathbb R$, the function $\eta_p$ such that
		$$\eta_{p}:(x,u)\mapsto\eta_{p}(x,u)=|u-w_{p}(x)|.$$ 
		This application verifies the inequality: 
		$$\partial_{t}\eta_{p}(u(t,x),x)+\dive (\sgn (u-w_{p})(f(u,x)-f(w_{p},x)))\leq \Delta \eta_{p}.$$
		
		In order to define our new entropy $\eta$, we define two auxiliary functions $p(u,x)$ and $\pi(x,t)$. We recall that for all $x\in\mathbb R^d$, the function $p\mapsto w_{p}(x)$ is a bijection from $\mathbb R$ to $\mathbb R$. We note  $p(u,x)$ the inverse of this application. It verifies: $$\forall x\in\mathbb R^d,u\in \mathbb R, w_{p(u,x)}(x)=u.$$ If $u$ is a function defined on $\mathbb R^+\times\mathbb R^d$ such that for all $(t,x)$, we define $\pi(t,x)=p(u(t,x),x).$ One remarks that $-R\leq\pi\leq R$. We can now define our particular entropy $\eta$ as:
		$$\eta(u,x)=\int_{0}^{p(u,x)}(u-w_{p}(x))dp.$$
		This function is non negative. Next, we derive energy estimate on $u$ using this new entropy. Deriving $\eta(u(t,x),x)$ with respect to $t$ and using (\ref{existence}), one obtains
		\begin{equation}\label{deriveta}\partial_{t}(\eta(u(t,x),x))\!=\!\int_{0}^{\pi(t,x)}\!\Delta(u-w_{p})dp-\int_{0}^{\pi(t,x)}\!\dive (f(u,x)-f(w_{p},x))dp.\end{equation}
		The last term of (\ref{deriveta}) is written as:
		$$\int_{0}^{\pi(t,x)}\dive (f(u,x)-f(w_{p},x))dp=\dive \left(\int_{0}^{\pi(t,x)}(f(u,x)-f(w_{p},x))dp\right)$$
		and
		$$\int_{0}^{\pi(t,x)}\Delta(u-w_{p})dp=\Delta(\eta(u(t,x),x))-\nabla\pi\cdot\nabla(u-w_{p})|_{p=\pi(t,x)}.$$
		We then obtain the following partial differential equation:
		\begin{equation}\label{a}\partial_{t}\eta(u)+\dive\!\left(\!\int_{0}^{\pi(t,x)}\!\!\!\!\!\!\!(f(u,x)\!-\!f(w_{p},x))dp\!\right)\!=\!\Delta\eta(u)-\!\nabla\pi\!\cdot\!\nabla(u-w_{p})|_{p=\pi(t,x)}.\end{equation}
		Moreover, we have the equality: \begin{equation}\label{b}0=\nabla(u(t,x)-w_{\pi(t,x)}(x))=\nabla(u-w_{p})|_{p=\pi(t,x)}-\partial_{p}w_{\pi}\cdot\nabla\pi.\end{equation}
		We deduce from (\ref{a}) and (\ref{b}) that $\eta$ satisfies the equation:
		\begin{equation}\label{c}\partial_{t}\eta(u)+\dive \left(\int_{0}^{\pi(t,x)}(f(u,x)-f(w_{p},x))dp\right)=\Delta\eta-\partial_{p}w_{\pi}\cdot|\nabla\pi|^2.\end{equation}
		
		Integrate equation (\ref{c}) in space: we get
		$$\frac d{dt}\int_{\mathbb R^d}\eta(u)(x)dx+\int_{\mathbb R^d}\partial_{p}w_{\pi}|\nabla\pi|^2=0.$$
		Moreover, theorem \ref{thmdalibard} gives us $\partial_{p}w_{\pi}\geq \alpha>0$.
		Using this inequality and Nash inequality (\cite{Taylor}): 
		$$\|\pi\|_{2}\leq c_{d}\|\pi\|_{1}^{(1-\theta)}\|\nabla\pi\|_{2}^\theta\quad\text{ where }\quad\frac1\theta=1+\frac2d,$$
		we obtain:
		\begin{equation}\label{4integrate}\frac d{dt}\int_{\mathbb R^d}\eta(u)(x)dx+C_{d}\frac{\|\pi\|_{2}^{2/\theta}}{\|\pi\|_{1}^{2(1-\theta)/\theta}}\leq0. \end{equation}
		
		Let us now relate $\pi$ with $\eta$:
		$$\eta(u(t,x),x)=\int_{0}^{\pi(t,x)}(u(t,x)-w_{p}(x))dp.$$
		From the estimate
		\begin{equation}\label{upi}\begin{array}{rclcc}|u(t,x)-w_{p}(x)|&=&|w_{\pi(t,x)}(x)-w_{p}(x)|&=&\left|\int_{p}^{\pi(t,x)}\partial_{p}w_{p}(x)dp\right|\\&\leq&|\pi(t,x)|\sup_{p}|\partial_{p}w_{p}|,\end{array}\end{equation}
		we deduce, $$\eta(u(t,x),x)\leq|\pi(t,x)|^2\sup_{p}|\partial_{p}w_{p}|.$$
		Since $\partial_{p}w_{p}$ is locally bounded in $p$, i.e. $\partial_{p}w_{p}(x)\leq C$ for all $x\in\mathbb R^d,$ for all $p\in[-R,R]$, we deduce the inequality:
		\begin{equation}\label{4etapi}\eta(u(t,x),x)\leq C|\pi(t,x)|^2.\end{equation}
		We combine (\ref{4integrate}) and (\ref{4etapi}) to obtain:
		$$\frac d{dt}\left(\int_{\mathbb R^d}\eta(u)(x)dx\right)+C\frac{(\int_{\mathbb R^d}\eta(u)(x)dx)^{1/\theta}}{\|\pi\|_{1}^{2(1-\theta)/\theta}}\leq0.$$
		We have now to overvalue $\|\pi\|_{1}$ uniformly in $t$. Now $$\pi(t,x)=p(u(t,x),x)-p(w_{0}(x),x)=\int_{w_{0}(x)}^{u(t,x)}\partial_{u}p(w,x)dw.$$
		We deduce from the minoration $\partial_{p}w_p\geq \alpha$ the estimate $\partial_{u}p\leq 1/\alpha$ and we deduce:
		$$\|\pi(t)\|_{1}\leq \frac1\alpha	\|u(t)-w_{0}\|_{1}\leq\frac1\alpha\|b\|_{1}$$
		with $L^1$-contraction. Finally, we have the inequation
		\begin{equation}\label{4inedp}\frac d{dt}\left(\int_{\mathbb R^d}\eta(u)(x)dx\right)+\frac{C}{\|b\|_{1}^{2(1-\theta)/\theta}}\left(\int_{\mathbb R^d}\eta(u)(x)dx\right)^{1/\theta}\leq0. \end{equation}
		
		Using $g:=-\left(\int_{\mathbb R^d}\eta(u)(x)dx\right)^{1-1/\theta}$, we solve this inequation and we obtain
		$$g(t)\leq (1-1/\theta)C\frac t{\|b\|_{1}^{2(1-\theta)/\theta}}.$$
		Hence,
		$$\left(\int_{\mathbb R^d}\eta(x)dx\right)\leq C'\frac{\|b\|_{1}^2}{t^{d/2}}.$$
		
		To conclude the proof, we prove that there exists $C>0$ such that for all $t\geq0$, $\sqrt{\int\eta(u(t))(x)}\geq C\|u(t)-w_{0}\|_{2}$. First, we have 
		$$\begin{array}{rcl}\eta(u)(x)&=&\int_0^{p(u(x),x)}(u(x)-w_{p}(x))dp\\&=&\int_0^{p(u(x),x)}(\int_p^{p(u(x),x)}\partial_pw_q(x)dq)dp\\&\geq&\alpha\int_0^{p(u(x),x)}(p(u(x),x)-p)dp\\&=&\alpha\frac{p(u(x),x)^2}{2}.\end{array}$$
		Then, estimate (\ref{upi}) shows that: 
		$$|u-w_0|^2\leq(\sup_p|\partial_pw_p|)^2p(u(x),x)^2\leq C^2p(u(x),x)^2.$$
		This concludes the proof of the theorem.
	\end{proof}
	
	\section{One-dimensional space case: proof of theorem \ref{thm}}\label{oneD}
	In this section, we prove $L^1$-convergence in one space dimension. We bypass the utilisation of Duhamel's formula by counting the number of sign changes. This method is used by Freist\"uhler and Serre in \cite{Freistuhler} to prove that constants are stable in $L^1$ when the flux $f$ does not depend on $x$, and when the space dimension is one. It uses a lemma of Matano \cite{Matano} which gives an evaluation of the number of sign changes for the solution of our equation. The proof is carried out in four steps: (1) At first, we make additional assumptions on $f$ and the initial datum. (2) Then, we prove $L^2$-estimates on $u$ and its primitive $V$ and we deduce that $\|V(t)\|_\infty$ vanishes at $+\infty$. (3) Using lemma of Matano, we find that $\|u(t)\|_1$ is controlled by $\|V(t)\|_\infty$, so we prove the result under the additional hypothesis. (4) We generalized the result without these assumptions.

	\begin{proof}
		First, up to a translation, we will assume that $$p=0,w_{p}\equiv0 \text{ and }f(0,\cdot)\equiv0.$$ We define $F(u,x)=f(u,x)-\partial_{u}f(0,x)u$ which verifies: $F(0,\cdot)\equiv0,$ and $ \partial_{u}F(0,\cdot)\equiv0$ and we deduce the inequality $$F(u,x)\leq \frac{|u|^2}2\sup |\partial_{u}^2F|.$$
		(1) Let us first assume that $b$ is bounded in the following sense: let $$p^+=\min\{p:b\leq w_{p}\}, p^-=\max\{p:b\geq w_{p}\},$$ we assume that $$\max\{\|w_{p^+}\|_{\infty},\|w_{p^-}\|_{\infty}\}<r.$$ Therefore, we have: $|b|<r$ and using the comparison property for all $t$, $|S^tb|<r.$ Moreover, we assume $\|b\|_{1}\sup_{[-r,r]}|\partial_{u}^2F|\leq1$. We will see at the end of the proof how to remove these assumptions.\\
			We further assume that $b\in\mathcal C_{0}^\infty(\mathbb R,[-r,r]),l(b)<\infty$ where $l(b)$ is the number of sign changes of $b$. Actually, we can approximate every function $b$ that verifies the conditions of step 1 by a function in $\mathcal C^\infty_{0}$; and since the support is compact, we can suppose that the sign of the function changes only a finite number of time.
			
		\noindent(2) Assume now that $b$ verifies all the previous assumptions. We define $V(x)=\int_{-\infty}^xu(t,y)dy.$ Since $u$ belongs to $L^1,$ $V$ is well defined and belongs to $L^\infty$ and $\|V\|_{\infty}\leq\|b\|_{1}$. Moreover, since $\int_\mathbb R b=0$ and we have mass conservation, we have that $V\in\mathcal C_0^\infty$. In search of estimates on $V$, we consider an equation verified by $V$:
			\begin{equation}\label{equation V}\partial_{t}V+\partial_{u}f(0,x)\partial_{x}V+F(\partial_{x}V,x)=\partial_{x}^2V.\end{equation}
			Let $\theta:x\mapsto\theta (x)$ from $\mathbb R$ to $\mathbb R$ be a positive function (which will be specified later). Multiplying by $\theta V$ and integrating in space, we obtain:
			$$\frac d{dt}\int\!\frac12\theta V^2+\!\int\!\theta|\partial_{x} V|^2=-\!\int\!\theta VF(\partial_{x}V,x)+\!\int\!\frac{V^2}2\left(\partial_{x}(\theta\partial_{u}f(0,x))\partial_{x}^2\theta\right).$$
			Besides, we have the inequality: $|F(\partial_{x}V,x)|\leq\frac{|\partial_{x}V|^2}2\sup|\partial_{u}^2F|$. We deduce the estimate:
			$$\frac d{dt}\left(\int \theta V^2\right)\leq-\int\theta|\partial_{x}V|^2+\int V^2\left(\partial_{x}(\theta\partial_{u}f(0,x))+\partial_x^2\theta\right).$$
			Now we choose $\theta$ to obtain an estimate on $\int\theta V^2$. We impose:
			\begin{itemize}
				\item $\theta>\alpha>0$ so that $V\mapsto\int\theta V^2$ is a norm on $L^2.$
				\item $\partial_{x}(\theta\partial_{u}f(0,x))+\partial_{x}^2\theta=0$.\\Actually, we only need that $\partial_{x}(\theta\partial_{u}f(0,x))+\partial_{x}^2\theta\leq0.$
			\end{itemize}
			The following lemma ensures the existence of such a $\theta$:
			\begin{lemme}
				There exists $\theta>0$ in $H^1_{per}(Y)$ such that $$\partial_{x}(\theta\partial_{u}f(0,x))+\partial_{x}^2\theta=0.$$
			\end{lemme}
			\begin{proof}
				We focus on the equation:
				$$\partial_{t}w-\partial_{x}(f(w,x))=\partial_{x}^2w.$$
				Theorem \ref{thmdalibard} ensures the existence of a periodic stationary solution $\tilde w_{p}$ of space average $p$ and this one verifies: $\partial_{p}\tilde w_{p}>0.$ Moreover, the function defined by $\theta\equiv\partial_{p}\tilde w_{p}|_{p=0}$ is $Y$-periodic, in $H^1$ and verifies the following equation:
				$$\partial_{x}(\theta\partial_{v}f(\tilde w_{0},x))+\partial_{x}^2\theta=0.$$
				We remark that $\partial_{x}f(0,x)=0=\partial_{x}^20.$ Since $\tilde w_0$ is the unique function such that $\partial_{x}^2\tilde w_{0}=-\partial_{x}f(\tilde w_{0},x)$ and $ \langle \tilde w_{0}\rangle_{Y}=0$, we have $\tilde w_{0}\equiv0$.
			\end{proof}
			The definition of $\theta$ ensures the inequality:
			$$\frac d{dt}\left(\int \theta V^2\right)\leq-\int\theta|\partial_{x}V|^2.$$
			Since $\theta$ belongs to $H^1_{per}(Y)\subset\mathcal C(\mathbb R)$, there exists $c>0$ such that $c<\theta$. Hence, we deduce that $V$ is bounded in $L^2(\mathbb R)$: 
			\begin{equation}\label{V2}
				c\int|V|^2(t)\leq\int\theta|V|^2(t)\leq\int \theta|V|^2(0).
			\end{equation}
			We also have an estimate on $\|u\|_2$. Indeed, we proved in section 4 the dispersion inequality (\ref{dispersion}) for $u$:
			$$\left(\int_{\mathbb R}|u(x,t)|^2dx\right)\leq C_{1}\frac{\|b\|_{1}^2}{t^{1/2}}.$$
			We deduce that 
			\begin{equation}\label{v2}
				\lim_{t\to\infty}\|u(t)\|_{2}=0.
			\end{equation}
			We can now prove an estimate on $\|V\|_\infty$. We have $$V^2(x,t)=2\int_{-\infty}^xu(y,t)V(y,t)dy\leq2\|u(\cdot,t)\|_{2}\|V(\cdot,t)\|_{2}.$$
			From equations (\ref{v2}) and (\ref{V2}), we deduce: $$\lim_{t\to\infty}\|u(\cdot,t)\|_{2}=0,\|V(\cdot,t)\|_{2}\text{ uniformly bounded in }t.$$
			Consequently $\displaystyle\lim_{t\to\infty}\|V(\cdot,t)\|_{\infty}=0.$
			
		\noindent(3) We now need an estimate on the number of sign changes of the solution $u$. To obtain it, we refer to the article of Matano \cite{Matano} in which an estimate on the lap number of a solution of a parabolic problem is proved.\\			
			Let $g:\mathbb R\to\mathbb R$ be a continuous function. We define its \emph{lap number} $l$ as the supremum of $0$ and all $k\in\mathbb N$ with the property that there exist $k+1$ points $x_{0}<\dots<x_{k}$ such that $$\forall 0<i<k, (g(x_{i+1})-g(x_{i}))(g(x_{i})-g(x_{i-1}))<0.$$
			We adapt the lemma of Matano \cite{Matano} to get:
			\begin{lemme}\label{matano}
				For any bounded solution $V:[0,\infty)\times\mathbb R\to\mathbb R$ of (\ref{equation V}):
				$$\partial_{t}V+\partial_{u}f(0,x)\partial_{x}V+F(\partial_{x}V,x)=\partial_{x}^2V$$
				with $V(0,\cdot)\in\mathcal C_{0}^\infty(\mathbb R)$ having a finite lap number, the lap number of $V(t,\cdot)$ is uniformly bounded for all $t\geq0.$
			\end{lemme}
			To do that, we just have to notice that $F(\partial_{x}V,x)=\tilde F(\partial_{x}V,x)\partial_{x}V$ with $\tilde F(\partial_{x}V,x)$. \\
			Since the number of sign changes of $b$ is finite, $V(0,x)$ has a finite lap number. The lemma of Matano proves that $\forall t,\exists\xi_{1}^t,\dots,\xi_{m}^t$ such that $V$ is monotone on $]-\infty=\xi_{0}^t;\xi_{1}^t[,\dots,$ $]\xi_{m}^t;\xi_{m+1}^t=\infty[.$ Therefore, the sign of $u$ does not change on the same intervals.
			We deduce: $$\begin{array}{rcl}\|u(\cdot,t)\|_{1}=\sum_{i=0}^m\left|\int_{\xi_{i}^t}^{\xi_{i+1}^t}u(x,t)dx\right|&=&\sum_{i=0}^m|V(\xi_{i+1}^t,t)-V(\xi_{i}^t,t)|\\&\leq &2(m+1)\|V(t)\|_{\infty}\to0.\end{array}$$
			Therefore the theorem is proved under the assumptions:
			$$\max\{\|w_{p^+}\|_{\infty},\|w_{p^-}\|_{\infty}\}<r,\quad\|b\|_{1}\sup_{[-r,r]}|\partial_{u}^2F|\leq1$$with$$p^+=\min\{p:b\leq w_{p}\},\quad p^-=\max\{p:b\geq w_{p}\}.$$

		\noindent(4) Next, we show how to remove these assumptions. We define
			$$A_{p}=\big\{b\in L^1(\mathbb R):\int_{-\infty}^\infty b=0 \text{ et }\forall x,w_{-p}(x)\leq b(x)\leq w_{p}(x)\big\}.$$
			We note $M_{p}=\max\{\|w_{-p}\|_{\infty},\|w_{p}\|_{\infty}\}.$ Hence, we have $$\sup_{[-M_{p},M_{p}]}|\partial_{u}^2F|<\infty.$$\\
			Let now $b\in A_{p}$ et $n=2\|b\|_{1}\sup_{[-M_{p},M_{p}]} |\partial_{v}^2F|.$ Using $w_{-p}\leq0\leq w_{p}$, we have $b/n\in A_{p}$ et $\forall k\in\{1,\dots,n\}, \frac{kb}n\in A_{p}$. The properties of the nonlinear semigroup show that $A_{p}$ is stable under $S^t,$  so we have for all $t, S^t(\frac{kb}n)\in A_{p}.$\\			
			By induction on $k$, we can prove the theorem for $\frac{kb}n.$ Let $P_{k}$ the property:
			$$P_{k}:\lim_{t\to\infty}\Big\|S^t\Big(\frac{kb}n\Big)\Big\|_{1}=0$$
			$P_{1}$: We have $b/n\in A_{p}, \|\frac bn\|_{1}\sup_{[-M_{p},M_{p}]} |\partial_{u}^2F|=\frac12<1.$ Using step 3, we deduce: $\lim_{t\to\infty}\Big\|S^t\Big(\frac{b}n\Big)\Big\|_{1}=0$.\\
			$P_k$: Let assume that $P_{k}$ with $k< n$ is true and let prove $P_{k+1}$. We have $S^t((k+1)\frac bn)\in A_{p}$. Moreover, the $L^1$-contraction property gives:
			$$\left\|S^t\left((k+1)\frac bn\right)-S^t\left(\frac{kb}n\right)\right\|_{1}\leq\left\|\frac bn\right\|_{1}.$$
			We deduce: $$\|S^t((k+1)\frac bn)\|_{1}\leq\|S^t(\frac{kb}n)\|_{1}+\|\frac bn\|_{1}.$$ Since $$\displaystyle\lim_{t\to\infty}\left\|S^t\left(\frac {kb}n\right)\right\|_{1}=0,$$ we have $$\|S^t((k+1)\frac bn)\|_{1}\sup_{[-M_{p},M_{p}]} |\partial_{u}^2F|<1$$ for $t$ large enough. Furthermore, $S^t((k+1)\frac bn)\in A_{p}.$ Hence, we can use the conclusion of step 3 again to conclude the proof.
	 \end{proof}
	 
	\section{Perspectives}
	
	In this paper we have proved the $L^1$-stability of the periodic stationary solutions of (\ref{equation}) in the one-dimensional space case. The proof uses a dispersion inequality which is also verified in the multidimension space case and the lemma of Matano (lemma \ref{matano}) about the number of sign changes of the solution of (\ref{equation}). But in the multidimension space case, the lemma of Matano has no more sense. An idea to bypass it is to use Duhamel's formula, as done by Serre in \cite{Serre}. In this purpose, we consider the linearized operator $L=\Delta-\dive(\partial_uf(0,x)\cdot)$, and we write the equation in the form: $$(\partial_t-L)u=-\dive(F(u,x))$$ with $F(u,x)=f(u,x)-\partial_{u}f(0,x)u$. We note $\tilde K^t$ the kernel of the operator $L$ so that we obtain Duhamel's formula: 
	$$u(t)=\tilde K^t*b-\int_{0}^t\nabla_{x}\tilde K^{t-s}*F(u(s,\cdot),\cdot)ds.$$
	Taking $L^1$-norms: 
	\begin{equation}\label{Duhamel}u(t)\leq\|\tilde K^t*b\|_1+\int_{0}^t\|\nabla_{x}\tilde K^{t-s}\|_1\|F(u(s,\cdot),\cdot)\|_1ds.\end{equation}
	Moreover, we have $\partial_{u}F(0,\cdot)\equiv0$, so we obtain $|F(u,\cdot)|\leq|u|^2$. Hence, dispersion inequality (\ref{dispersion}) gives: $$\|F(u(s,\cdot),\cdot)\|_1\leq C_d^2\frac{\|b\|_1^2}{s^{d/2}}.$$

	To obtain an $L^1$-convergence theorem similar to theorem \ref{thm}, we can use estimates on the kernel $\tilde K^t$ and its derivative $\nabla_x\tilde K^t$. Some results on this kernel are given by Oh and Zumbrun in \cite{Oh1} and \cite{Oh} when the space dimension is one. When the space dimension $d$ is larger than 2, we can refer to \cite{Oh3} and \cite{Oh4} in which they obtain large-time estimates in $L^q$ where $q\geq2$, and when $f$ is periodic in only one direction. But, until now, we have not large-time $L^1$-estimates for $d\geq2$.

	To conclude, we can see how estimates can give a theorem: if we obtain suitable estimates, we can bound all the term in (\ref{Duhamel}) by $\|b\|_1^2$ as in \cite{Serre} and conclude as Serre does by continuity of the limit: $l_0(b)=\displaystyle\lim_{t\to\infty}\|S^tb\|_1.$

	\paragraph*{Acknowledgement.}  The author wishes to thanks Denis Serre for suggesting the problem and Pascal Noble for useful comments on the manuscript.
		
	\bibliographystyle{plain}
	\bibliography{biblio}

\end{document}